\documentclass[12pt]{article}

\usepackage{amssymb,amsmath,amsbsy, amsthm}
\usepackage[english]{babel}
\usepackage{amscd}
\usepackage{pb-diagram}
\usepackage{pstricks,pst-node}
\usepackage{amsfonts}
\usepackage{latexsym}
\usepackage{pstricks,pst-node}
\usepackage{tikz}
\usepackage{fullpage}
\usepackage{color}

\usepackage{hyperref}
\usepackage[all]{xy}

\usepackage{float}
\usepackage{lineno}

\newtheorem{theorem}{Theorem}[section]
\newtheorem{lemma}[theorem]{Lemma}

\newtheorem{claim}[theorem]{Claim}

\newtheorem{proposition}[theorem]{Proposition}

\newtheorem{question}[theorem]{Question}

\theoremstyle{definition}
\newtheorem{definition}[theorem]{Definition}
\newtheorem{remark}[theorem]{Remark}

\numberwithin{equation}{section}


\newcommand{\N}{\mathbb{N}}

\newcommand{\Bd}{\mathrm{Bd}}
\newcommand{\Int}{\mathrm{Int}}

\newcommand{\Cl}{\mathrm{Cl}}

\newcommand{\NB}{\mathcal{NB}}
\newcommand{\Ha}{\mathcal{H}}

\begin{document}
\date{}
\title{Nonblockers for hereditarily decomposable continua with the property of Kelley}       
\author{Javier Camargo and Mayra Ferreira\footnote{2010 {\it Mathematics Subject Classification.} 54B20, 54F15.\hfill\break {\it Key word and phrases.} Continuum, Hyperspaces of continua, Nonblockers, Property of Kelley, Hereditarily decomposable continua.}}

\maketitle

\begin{abstract}
     Given a continuum $X$, let $\NB (\mathcal{F}_1(X))$ be the hyperspace of nonblockers of $\mathcal{F}_1(X)$.
    In this paper, we show that if $X$ is hereditarily decomposable with the property of Kelley such that $\NB (\mathcal{F}_1(X))$ is a continuum, then $X$ is a simple closed curve. Thus, we characterize the simple closed curve as the unique hereditarily decomposable continuum with the property of Kelley $X$ such that its hyperspace $\NB (\mathcal{F}_1(X))$ is a continuum.
\end{abstract}

\section{Introduction}

A continuum is a nonempty compact connected metric space. Given a continuum $X$, $2^X$ denotes the hyperspace of all nonempty closed subsets of $X$ topologized with the Hausdorff metric. The set $\mathcal{F}_1(X)$ denotes the family of all one-point subsets of $X$. 

Let $X$ be a continuum and let $B\in 2^X$. We say that $B$ does not block $\mathcal{F}_1(X)$ provided that for any $x\in X\setminus B$, the union of all subcontinua of $X$ containing $x$ and contained in $X\setminus B$ is a dense subset of $X$. The collection of all $B$ in $2^X$ such that $B$ does not block $\mathcal{F}_1(X)$ is denoted by $\NB (\mathcal{F}_1(X))$.

The notion of nonblocker is introduced by A. Illanes and P. Krupski in \cite{Illanes-Krupski}. Since $\NB (\mathcal{F}_1(X))\subseteq 2^X$, $\NB (\mathcal{F}_1(X))$ has the subspace topology and, it is natural to ask, over what conditions is the hyperspace $\NB (\mathcal{F}_1(X))$ a continuum? 

In \cite[Theorem 4.4]{Escobedo-Villanueva}, it is proved that if $X$ is locally connected, then, $\NB (\mathcal{F}_1(X))$ is a continuum if and only if $X$ is a simple closed curve. Also, in \cite{Escobedo-Villanueva}, the authors showed that $\NB (\mathcal{F}_1(X))$ is a continuum when $X$ is the circle of pseudoarcs. In \cite{Camargo-Ortiz}, the simple closed curve is characterized as the unique continuum $X$ such that the hyperspace $\NB (\mathcal{F}_1(X))$ is homeomorphic to $\mathcal{F}_1(X)$. In \cite{Camargo-Maya}, the authors show examples of continua $X$ such that $\NB(\mathcal{F}_1(X))$ is a continuum. Recently, A. Illanes and B. Vejnar showed in \cite{Illanes} that for any completely metrizable and separable space $X$, there exists a continuum $Y$ such that $\NB (\mathcal{F}_1(Y))$ is homeomorphic to $X$. 



The following question was taken from \cite[Question 4.5]{Escobedo-Villanueva}.

\begin{question}\label{Question1}
For which nonlocally connected continua $X$ is $\NB (\mathcal{F}_1(X))$ a continuum?
\end{question}

We note that in all the examples presented in \cite{Camargo-Maya}, \cite{Escobedo-Villanueva} or \cite{Illanes}, if $X$ is not a simple closed curve and $\NB (\mathcal{F}_1(X))$ is a continuum, then $X$ contain an infinite number of indecomposable continua. Hence, we rise the following question:

\begin{question}\label{Question}
Let $X$ be an hereditarily decomposable continuum such that $\NB (\mathcal{F}_1(X))$ is a continuum. Then, does it follow that $X$ is the simple closed curve? 
\end{question}

The main result in this paper is Theorem~\ref{theo1q2q3q}, where we characterize the simple closed curve as the unique hereditarily decomposable continuum with the property of Kelley $X$, such that the hyperspace of nonblockers $\NB (\mathcal{F}_1(X))$ is a continuum.

The reader can see \cite{Bobok}, \cite{Macias}, or \cite{Piceno} for additional information about nonblockers in hyperspaces.

\section{Preliminaries}

The symbols $\mathbb{N}, \mathbb{R}$ and $\mathbb{C}$ denote the set of positive integers, reals numbers and complex numbers, respectively. In this paper, every map will be a continuous function. Given a metric space $X$ and  $A\subseteq X$, $\Int (A)$, $\Cl(A)$ and $\Bd(A)$ denote the interior, the closure and the boundary of $A$, respectively. For $x\in X$ and $r>0$, we represent $B(x,r)=\{y\in X : d(x,y)<r\}$ where $d$ is a metric on $X$. A \textit{continuum} is a nonempty compact connected metric space. A \textit{subcontinuum} is a continuum contained in some metric space. We say that a continuum $X$ is \textit{irreducible between two of its points} if no proper subcontinuum of $X$ contains both points. A continuum is \textit{irreducible} if it is irreducible between two of its points. A subcontinuum $Z$ of a continuum $X$ is said to be \textit{terminal} if each subcontinuum $Y$ of $X$ that intersects $Z$ satisfies either $Z\subseteq Y$ or $Y\subseteq Z$.
An \textit{arc} is any continuum homeomorphic to $[0,1]$, and a \textit{simple closed curve} is any continuum homeomorphic to $S^1=\{z\in\mathbb{C} : |z|=1\}$.  

A continuum is \textit{decomposable} provided that it is the union of two of its proper subcontinua. A continuum that is not decomposable is said to be \textit{indecomposable}. A continuum is \textit{hereditarily decomposable} if every subcontinuum is decomposable.

Let $X$ be a continuum and let $\mathcal{D}$ be a decomposition of $X$. We say that $\mathcal{D}$ is an \textit{upper semicontinuous decomposition} provided that whenever $D\in\mathcal{D}$, $U$ open subset of $X$, and $D\subseteq U$, there exists an open subset $V$ of $X$ with $D\subseteq V$ such that if $A\in\mathcal{D}$ and $A\cap V\neq\emptyset$, then $A\subseteq U$. We say that $\mathcal{D}$ is \textit{lower semicontinuous} provided that for each $D\in \mathcal{D}$ any two points $x$ and $y$ of $D$ and each open subset $U$ of $X$ such that $x\in U$, there exists an open subset $V$ of $X$ such that $y\in V$ and such that if $A\in \mathcal{D}$ and $A\cap V\neq\emptyset$, then $A\cap U\neq\emptyset$. Finally, we say that $\mathcal{D}$ is \textit{continuous} if $\mathcal{D}$ is both upper and lower semicontinuous. If $\mathcal{D}$ is a decomposition of $X$, we denote $\rho\colon X\to \mathcal{D}$ by the natural function such that $\rho(x)=D_x$ where $D_x$ is unique $D_x\in\mathcal{D}$ with $x\in D_x$. The following result is not difficult to prove.

\begin{proposition}\label{usc}
Let $X$ be a continuum, let $\mathcal{D}$ be a decomposition of $X$ such that $D$ is closed for each $D\in\mathcal{D}$. Then, $\mathcal{D}$ is upper semicontinuous (lower semicontinuous) if and only if for each sequence $(x_n)_{n\in\mathbb{N}}$ in $X$ such that $\lim_{n\to\infty}x_n=x$ for some $x\in X$, we have that $\limsup \rho(x_n)\subseteq \rho(x)$ ($\rho(x)\subseteq \liminf \rho(x_n)$, respectively).
\end{proposition}

Given a map between continua $f\colon X\to Y$, we say that $f$ is \textit{monotone} provided that $f^{-1}(D)$ is a continuum for each subcontinuum $D$ of $Y$. We say that $f$ is \textit{open} provided that $f(U)$ is open for every open subset $U$ of $X$. 

\begin{definition}\label{defaposyndetic}
Let $X$ be a continuum, and let $p,q\in X$. We say that $X$ is \textit{aposyndetic at $p$ with respect to $q$} provided that there exists a subcontinuum $W$ of $X$ such that $p\in \Int(W)\subseteq W\subseteq X\setminus\{q\}$. Now, $X$ is \textit{aposyndetic at $p$} if $X$ is aposyndetic at $p$ with respect to each point of $X\setminus\{p\}$. We say that $X$ is \textit{aposyndetic} provided that $X$ is aposyndetic at each of its points.
\end{definition}

A continuum $X$ is \textit{locally connected} provided that for every $p\in X$ and each open $V$ such that $p\in V$, there exists an open connected set $U$ such that $p\in U\subseteq V$. We say that $X$ is \textit{semi-locally connected} provided that for every $p\in X$ and each open $V$ such that $p\in V$, there exists an open $U$ such that $p\in U\subseteq V$ and $X\setminus U$ has a finite number of components.

\medskip

Let $X$ be a continuum. Let $2^X=\{A\subseteq X : A\neq\emptyset\text{ and }\Cl (A)=A\}$. It is known that the collection of sets $\langle U_1,U_2,...,U_l\rangle$ form a base on $2^{X}$ (\textit{Vietoris topology}, see \cite[p.3]{Illanes-Nadler}), where $U_1,U_2,...,U_l$ are open sets in $X$ and
$$\langle U_1,U_2,...,U_l\rangle=\{A\in 2^{X}:A\subseteq \cup_{i=1}^{l}U_i\text{  and }A\cap U_i\neq\emptyset \text{ for each } i\}.$$
We denote $\mathcal{C}(X)=\{A\in 2^X : A \text{ is connected}\}$ and $\mathcal{F}_1(X)=\{\{x\} : x\in X\}$. Since $\mathcal{F}_1(X)\subseteq \mathcal{C}(X)\subseteq 2^X$, the sets $\mathcal{F}_1(X)$ and $\mathcal{C}(X)$ are considered as subspaces of $2^X$. Let $\Ha$ be defined for each $A,B\in 2^X$ by $$\Ha(A,B)=\inf\{r>0 : A\subseteq N(B;r) \text{ and }B\subseteq N(A;r)\},$$ where $N(D,s)=\{x\in X : d(x,y)<s \text{ for some }y\in D\}$, for any $D\in 2^X$ and $s>0$. $\Ha$ is a metric on $2^X$ and it is known as the \textit{Hausdorff metric} \cite[Theorem~0.2]{Nadler-Libro2}. It is well known that the topology induced by the Hausdorff metric and the Vietoris topology are the same \cite[Theorem~3.1]{Illanes-Nadler}.

Let $A,B\in 2^X$ such that $A\cap B=\emptyset$. We use the notation $$\kappa_{X\setminus A}(B)=\bigcup \{L\in\mathcal{C}(X) : B\cap L\neq\emptyset\text{ and }L\subseteq X\setminus A\}.$$ If $B=\{b\}$ for some $b\in X$, we write $\kappa_{X\setminus A}(b)$ instead of $\kappa_{X\setminus A}(\{b\})$.

\begin{lemma}\label{lemma0}
Let $X$ be a continuum and let $A,B\in 2^X$ such that $A\cap B=\emptyset$. If $B$ is connected, then there exists a sequence of subcontinua $(K_n)_{n\in\mathbb{N}}$ of $X$ such that $\kappa_{X\setminus A}(B)=\bigcup_{n\in\mathbb{N}}K_n$.
\end{lemma}

\begin{proof}
Let $U_n=N(A,1/n)$ for each $n\in\mathbb{N}$. Since $A$ and $B$ are compact and $A\cap B=\emptyset$, we may suppose that $U_n\cap B=\emptyset$ for every $n\in\mathbb{N}$. Let $K_n$ be the component of $X\setminus U_n$ such that $B\subseteq K_n$ for each $n\in\mathbb{N}$. It is not difficult to see that $\kappa_{X\setminus A}(B)=\bigcup_{n\in \mathbb{N}}K_n$.
\end{proof}

A subset $W$ of a continuum $X$ is said to be \textit{connected by continua} provided for any two points $x,y\in W$ there exists a subcontinuum $L$ of $W$ such that $\{x,y\}\subseteq L\subseteq W$. Observe that for $A\in 2^X$, $X\setminus A$ is connected by continua if $\kappa_{X\setminus A}(x)=X\setminus A$ for every $x\in X\setminus A$. 

\begin{definition}
Let $X$ be a continuum and let $A,B\in 2^X$ such that $A\cap B=\emptyset$. We say that \textit{$A$ does not block $B$} provided that $\kappa_{X\setminus A}(B)$ is a dense subset of $X$. In another way we say that \textit{$A$ blocks $B$}; i.e., $A$ blocks $B$ provided that $\Cl(\kappa_{X\setminus A}(B))\neq X$.
\end{definition}
In this paper, we will focus on the following set:
$$\NB(\mathcal{F}_1(X))=\{A\in 2^X : A \text{ does not block }\{x\} \text{ for each }x\in X\setminus A\}.$$

The sets $\mathcal{F}_1(X), \mathcal{C}(X), \NB(\mathcal{F}_1(X))$ and $2^X$ are called \textit{hyperspaces} of $X$. In particular, $\NB(\mathcal{F}_1(X))$ is called \textit{the hyperspace of nonblockers of $\mathcal{F}_1(X)$}. 
In \cite[Theorem 4.4]{Escobedo-Villanueva}, it is proved the following theorem.

\begin{theorem}\label{locallyconnected}
Let $X$ be a locally connected continuum. Then, $\NB(\mathcal{F}_1(X))$ is a continuum if and only if $X$ is a simple closed curve.
\end{theorem}



Let $X$ be a continuum. We say that $X$ has the \textit{property of Kelley} provided that for any sequence $(x_n)_{n\in\mathbb{N}}$ such that $\lim_{n\to\infty}x_n=x$ for some $x\in X$, and for any $A\in\mathcal{C}(X)$ such that $x\in A$, there exists a sequence on subcontinua $(A_n)_{n\in\mathbb{N}}$ of $X$ such that $x_n\in A_n$ for each $n\in\mathbb{N}$, and $\lim_{n\to\infty}A_n=A$.

The proof of the next result can be found in \cite[Corollary 4.2]{Macias1}.

\begin{theorem}\label{theokelley}
Let $X$ be a continuum, let $A\in 2^X$ and let $K$ be a subcontinuum of $X\setminus A$ such that $\Int (K)\neq\emptyset$. If $X$ has the property of Kelley, then there exists a subcontinuum $L$ of $X$ such that $K\subseteq \Int (L)\subseteq L\subseteq X\setminus A$.
\end{theorem}

\section{On minimal nonblockers}

In \cite{Piceno}, it is introduced the function $\pi\colon X\to \mathcal{C}(X)$ defined by

\begin{equation}\label{defminimal}
 \pi(x)=    
    \begin{cases}
    \bigcap \{A\in\NB(\mathcal{F}_1(X)) : x\in A\}, & \hbox{ if } x\in \bigcup\mathcal{NB}(\mathcal{F}_1(X)); \\
    X, &\hbox{ if } x\notin \bigcup\mathcal{NB}(\mathcal{F}_1(X)).
    \end{cases}   
\end{equation}

By Proposition~3.1 and Theorem~3.3 of \cite{Piceno}, $\pi$ is well defined and $\pi(x)\in \NB(\mathcal{F}_1(X))$ for each $x\in X$ where $x\in \bigcup\mathcal{NB}(\mathcal{F}_1(X))$. We prove in Theorem~\ref{partition} that $\pi$ is a map whenever $X$ is a decomposable continuum with the property of Kelley such that $\NB(\mathcal{F}_1(X))$ is a continuum, giving a partial answer to Question~3.6 of \cite{Piceno}. Furthermore, we show in Theorem~\ref{theopartition} that $\{\pi(x) : x\in X\}$ is a continuous decomposition of $X$.

\begin{lemma}\label{lemdyfh6y}
Let $X$ be a decomposable continuum with the property of Kelley. If $L$ is a proper subcontinuum of $X$, then there exists $A\in\NB(\mathcal{F}_1(X))$ such that $A\subseteq X\setminus L$.
\end{lemma}

\begin{proof}
Let $L$ be a proper subcontinuum of $X$ and let $M$ and $N$ be proper subcontinua of $X$ such that $X=M\cup N$. We consider two cases:

\textbf{1.} $\Int(L)=\emptyset$. Note that either $M\cap L\neq\emptyset$ or $N\cap L\neq \emptyset$. Suppose that $M\cap L\neq\emptyset$. Since $X\setminus M$ is a nonempty open subset of $X$ and $X\setminus M\nsubseteq L$, we have that $M\cup L\neq X$. Let $L_1=M\cup L$. Since $\Int(M)\neq\emptyset$, $\Int(L_1)\neq\emptyset$. By Theorem~\ref{theokelley}, there exists a proper subcontinuum $R$ of $X$ such that $L_1\subseteq \Int(R)$. By \cite[Theorem~14.6]{Illanes-Nadler}, there exists an order arc in $\mathcal{C}(X)$ from $R$ to $X$. Let $\alpha$ be an order arc from $R$ to $X$ and let $f\colon [0,1]\to \mathcal{C}(X)$ be an embedding such that $f(0)=R, f(1)=X$ and $f([0,1])=\alpha$. Since $f$ is a map, there exists $t\in (0,1)$ such that $\Ha(f(t),X)<\frac{1}{2}$. Let $L_2=f(t)$. Note that $L_1\subseteq \Int(L_2)$. Hence inductively, we may construct an increasing sequence of continua $(L_n)_{n\in\mathbb{N}}$ such that $L_n\subseteq \Int(L_{n+1})$, $L_n\neq X$ and $\Ha(L_n,X)<\frac{1}{n}$, for each $n\in\mathbb{N}$. Let $A=\bigcap_{n\in\mathbb{N}}(X\setminus \Int(L_n))$. Since $\Int(L_n)\subseteq \Int(L_{n+1})$, $X\setminus \Int(L_{n+1})\subseteq X\setminus \Int(L_{n})$, for each $n\in\mathbb{N}$. Furthermore, $X\setminus \Int(L_n)\neq\emptyset$ for each $n\in\mathbb{N}$. Thus, $A$ is a nonempty compact subset of $X$. Observe that $X\setminus A=\bigcup_{n\in\mathbb{N}}\Int(L_n)$ and, since $\Int(L_n)\subseteq L_{n+1}$, $X\setminus A=\bigcup_{n\in\mathbb{N}}L_n$. Also, $X\setminus A$ is dense, because $\Ha(L_n,X)<\frac{1}{n}$ for each $n\in\mathbb{N}$. Therefore, $A\in\NB(\mathcal{F}_1(X))$ and $A\subseteq X\setminus L$. Similarly, we construct $A$ if $N\cap L\neq\emptyset$.

\textbf{2.} $\Int(L)\neq\emptyset$. We define $L_1=L$ and continue the same argument as in case \textbf{1} to construct $A$ and complete the proof. 
\end{proof}

With the next theorem, we have that $\pi(x)\neq X$ for each $x\in X$ when $X$ is a decomposable continuum with the property of Kelley such that $\NB(\mathcal{F}_1(X))$ is a continuum.

\begin{theorem}\label{theo98iw}
Let $X$ be a decomposable continuum with the property of Kelley. If $\NB(\mathcal{F}_1(X))$ is a continuum, then $X=\bigcup\NB(\mathcal{F}_1(X))$.
\end{theorem}

\begin{proof}
By \cite[Theorem~3.3]{Piceno}, $\NB(\mathcal{F}_1(X))\cap \mathcal{C}(X)\neq\emptyset$. Since $\NB(\mathcal{F}_1(X))$ is a continuum, $\bigcup\NB(\mathcal{F}_1(X))$ is a subcontinuum of $\mathcal{C}(X)$, by \cite[Exercise~15.9 (2)]{Illanes-Nadler}. Suppose that $\bigcup\NB(\mathcal{F}_1(X))$ is a proper subcontinuum of $\mathcal{C}(X)$. Thus, there exists $A\in\NB(\mathcal{F}_1(X))$ such that $A\subseteq X\setminus \bigcup\NB(\mathcal{F}_1(X))$, by Lemma~\ref{lemdyfh6y}. A contradiction. Therefore, $X=\bigcup\NB(\mathcal{F}_1(X))$.
\end{proof}

In this section we use the following notation: $$B(x)=\{y\in X : x \text{ blocks }y\}=\{y\in X : \Cl(\kappa_{X\setminus\{x\}}(y))\neq X\}, \text{ for }x\in X.$$ Observe that $X\setminus \pi(x)\subseteq X\setminus B(x)$ for each $x\in X$. Thus, $B(x)\subseteq \pi(x)$ for each $x\in X$. The following remark is not difficult to prove.

\begin{remark}\label{rem6676}
Let $X$ be a continuum and let $x\in X$. If $L\in\mathcal{C}(X)$ is such that $L\cap B(x)\neq\emptyset$ and $L\cap (X\setminus B(x))\neq\emptyset$, then $x\in L$.
\end{remark}

\begin{proposition}\label{blockpi}
Let $X$ be a decomposable continuum such that $\NB(\mathcal{F}_1(X))$ is a continuum. If $X$ has the property of Kelley, then  $\pi(x)=B(x)$ for each $x\in X$.
\end{proposition}
\begin{proof}
Let $x\in X$. Note that there exists $A\in\NB(\mathcal{F}_1(X))$ such that $x\in A$, by Theorem~\ref{theo98iw}. Thus, $\pi(x)\neq X$. Since $B(x)\subseteq \pi(x)$, we need to show that $\pi(x)\subseteq B(x)$. Observe that if $B(x)$ is closed, then $B(x)\in \NB(\mathcal{F}_1(X))$, and hence, $B(x)\in \{A\in\NB(\mathcal{F}_1(X)) : x\in A\}$ and $\pi(x)\subseteq B(x)$ (see (\ref{defminimal})). 

We see that $B(x)$ is closed. Let $(w_n)_{n\in\mathbb{N}}$ be a sequence in $B(x)$ such that $\lim_{n\to\infty}w_n=w$ for some $w\in X$. Suppose that $w\in \Cl(B(x))\setminus B(x)$. Since $B(x)\subseteq \pi(x)$, $\pi(x)$ is closed and $\Int(\pi(x))=\emptyset$, we have that there exists a nonempty open subset $U$ of $X$ such that $U\cap \Cl(B(x))=\emptyset$.
Furthermore, there exists $L\in\mathcal{C}(X)$ such that $w\in L$, $x\notin L$ and $L\cap U\neq\emptyset$, because $w\notin B(x)$. Since $X$ has the property of Kelley, there exists a sequence $(L_n)_{n\in\N}$ of $\mathcal{C}(X)$ such that $w_n\in L_n$ for each $n\in\N$, and $\lim_{n\to\infty}L_n=L$. Note that $L\in \langle X\setminus\{x\},U\rangle$. Thus, there exists a positive number $k\in\mathbb{N}$ such that $L_i\in \langle X\setminus\{x\},U\rangle$ for each $i\geq k$. Let $k_0>k$. Note that $L_{k_0}\cap B(x)\neq\emptyset$, $L_{k_0}\cap U\neq\emptyset$ and $x\notin L_{k_0}$. This contradicts Remark~\ref{rem6676}. Therefore, $B(x)$ is closed and $\pi(x)=B(x)$.
\end{proof} 

\begin{question}
    Let $X$ be a decomposable continuum. If $X$ has the property of Kelley, then does it follow that $B(x)$ is closed for each $x\in X$?
\end{question}

It is easy to verify that $B(x)$ is one composant of $X$, whenever $X$ is an hereditarily indecomposable continuum. Thus, if $X$ is hereditarily indecomposable, then $X$ has the property of Kelley and $B(x)$ is not closed (see \cite[20.6, p.168]{Illanes-Nadler}).



\begin{theorem}\label{partition}
Let $X$ be a decomposable continuum with the property of Kelley. If $\NB(\mathcal{F}_1(X))$ is a continuum, then $\pi$ is a map.
\end{theorem}

\begin{proof}
Let $(x_n)_{n\in\mathbb{N}}$ be a sequence in $X$ such that $\lim_{n\to\infty}x_n=x$ for some $x\in X$. We prove that $\lim_{n\to\infty}\pi(x_n)=\pi(x)$. To see this, we show that $\limsup \pi(x_n)\subseteq \pi(x)$ and $\pi(x)\subseteq \liminf \pi(x_n)$.

\medskip

We see that $\limsup \pi(x_n)\subseteq \pi(x)$. Let $z\in \limsup \pi(x_n)$. Let $(n_k)_{k\in\N}$ be an increasing sequence of positive numbers such that, for each $k\in\N$, there exists $z_{n_k}\in \pi(x_{n_k})$ where $\lim_{k\to\infty}z_{n_k}=z$. Furthermore, since $\NB(\mathcal{F}_1(X))$ is compact and $\pi(x_{n_k})\in\NB(\mathcal{F}_1(X))$ for all $k$, without loss of generality we may suppose that  $\lim_{k\to\infty}\pi(x_{n_k})=D$ for some $D\in\NB(\mathcal{F}_1(X))$. Hence, $\Int(D)=\emptyset$. 
Note that $z\in D$. Suppose that $z\in X\setminus \pi(x)$. Let $U$ and $V$ be nonempty open subsets of $X$ such that $D\subseteq U$ and $U\cap V=\emptyset$. Since $D\in\langle U\rangle$, there exists $m\in\N$ such that $\pi(x_{n_k})\in \langle U\rangle$ for each $k\geq m$; i.e., $\pi(x_{n_k})\subseteq U$ for each $k\geq m$. Also, since $z\notin \pi(x)$, there exists a subcontinuum $L$ of $X$ such that $z\in L, x\notin L$ and $L\cap V\neq\emptyset$. Observe that $\langle X\setminus\{x\}, V\rangle$ is an open subset of $2^X$ such that $L\in \langle X\setminus\{x\}, V\rangle$. Since $X$ has the property of Kelley, there exists a sequence of subcontinua $(L_{n_k})_{k\in\mathbb{N}}$ of $X$ such that $z_{n_k}\in L_{n_k}$ for each $k\in\mathbb{N}$, and $\lim_{k\to\infty} L_{n_k}=L$. Thus, there is $l\in\N$ such that $L_{n_k}\in \langle X\setminus\{x\}, V\rangle$ for each $k\geq l$. Let $j>\max\{m,l\}$. Since $\pi(x_{n_{j}})\subseteq U$, $L_{n_j}\cap V\neq\emptyset$ and $U\cap V=\emptyset$, we have that $L_{n_{j}}\cap X\setminus \pi(x_{n_{j}})\neq\emptyset$. Furthermore, $L_{n_{j}}\cap \pi(x_{n_{j}})\neq\emptyset$, because $z_{n_{j}}\in L_{n_{j}}\cap \pi(x_{n_{j}})$. By Proposition~\ref{blockpi}, $\pi(x_{n_{j}})=B(x_{n_{j}})$. Hence, $x_{n_{j}}\in L_{n_{j}}$, by Remark~\ref{rem6676}. Since $j$ was an arbitrary integer greater than $\max\{m,l\}$, we have that $x_{n_j}\in L_{n_j}$ for each $j\geq \max\{m,l\}$. Thus, $x\in L$. A contradiction. Therefore, $z\in \pi(x)$ and $\limsup \pi(x_n)\subseteq \pi(x)$.

\medskip

We show that $\pi(x)\subseteq \liminf \pi(x_n)$. Suppose that there exists $y\in\pi(x)$ such that $y\in X\setminus \liminf \pi(x_n)$. Then there exist an open subset $U$ of $X$ such that $y\in U$, and an increasing sequence of positive integers $(n_k)_{k\in\N}$ such that $\pi(x_{n_k})\cap U=\emptyset$ for each $k\in\N$. Since $\NB(\mathcal{F}_1(X))$ is compact and $(\pi(x_{n_k}))_{k\in\N}$ is a sequence in $\NB(\mathcal{F}_1(X))$, we may suppose that $\lim_{k\to\infty}\pi(x_{n_k})=D$ for some $D\in \NB(\mathcal{F}_1(X))$. Observe that $D\cap U=\emptyset$ and $x\in D$. This contradicts that $y\in\pi(x)$, $D\in\NB(\mathcal{F}_1(X))$ and $y\notin D$. Therefore, $\pi(x)\subseteq \liminf \pi(x_n)$.
\end{proof}

\begin{lemma}\label{orden}
Let $X$ be a decomposable continuum with the property of Kelley  such that $\NB(\mathcal{F}_1(X))$ is a continuum. Let $x,y\in X$. If $\pi(x)\cap\pi(y)\neq\emptyset$, then $\pi(x)\subseteq \pi(y)$ or $\pi(y)\subseteq \pi(x)$.
\end{lemma}
\begin{proof}
Let $x,y\in X$ such that $\pi(x)\cap\pi(y)\neq\emptyset$. Suppose that $\pi(x)\setminus \pi(y)\neq\emptyset$. Observe that $\pi(y)=B(y)$, by Proposition~\ref{blockpi}. Furthermore, $\pi(x)\in\mathcal{C}(X)$ such that $\pi(x)\cap B(y)\neq\emptyset$ and $\pi(x)\cap X\setminus B(y)\neq\emptyset$. Thus, $y\in \pi(x)$, by Remark~\ref{rem6676}. Therefore, $\pi(y)\subseteq \pi(x)$, by (\ref{defminimal}).
\end{proof}

Given a decomposable continuum with the property of Kelley $X$ such that $\NB(\mathcal{F}_1(X))$ is a continuum. Let 
\begin{equation}\label{defL}
    \mathcal{L}=\{\pi(x) : x\in X\}.
\end{equation} 
Note that $\mathcal{L}=\pi(X)$. Hence, $\mathcal{L}$ is a subcontinuum of $\NB(\mathcal{F}_1(X))$, by Theorem~\ref{partition}. Since $\mathcal{L}$ is compact, there are maximal and minimal elements of $\mathcal{L}$ that we represent:
\begin{equation}\label{defMN}
    \mathcal{M}=\{\pi(x) : \pi(x) \text{ is maximal in }\mathcal{L}\}; \text{ and }\mathcal{N}=\{\pi(x) : \pi(x) \text{ is minimal in }\mathcal{L}\}.
\end{equation}

\begin{proposition}\label{khig788tj}
Let $X$ be a decomposable continuum with the property of Kelley such that $\NB(\mathcal{F}_1(X))$ is a continuum. Let $\mathcal{L}, \mathcal{M}$ and $\mathcal{N}$ be as (\ref{defL}) and (\ref{defMN}). Then,
\begin{enumerate}
    \item $\mathcal{M}$ is an upper semicontinuous decomposition of $X$.
    \item $\pi(x)$ is terminal, for each $\pi(x)\in \mathcal{N}$.
    \item $\mathcal{N}$ is compact.
    \item $\bigcup\mathcal{N}$ is a subcontinuum of $X$.
\end{enumerate}
\end{proposition}

\begin{proof}
We see that $X=\bigcup\mathcal{M}$. Let $x\in X$. By (\ref{defminimal}), $x\in\pi(x)$. Furthermore, since $\mathcal{L}$ is compact, there exists $z\in X$ such that $\pi(x)\subseteq \pi(z)$ and $\pi(z)\in\mathcal{M}$. Therefore, $x\in \bigcup\mathcal{M}$ and $X=\bigcup\mathcal{M}$. Note that if $\pi(x)\cap\pi(y)\neq\emptyset$ and $\pi(x), \pi(y)\in \mathcal{M}$, then $\pi(x)=\pi(y)$, by Lemma~\ref{orden}. Thus, $\mathcal{M}$ is a decomposition of $X$. We show that $\mathcal{M}$ is an upper semicontinuous decomposition. Let $\rho\colon X\to \mathcal{M}$ be the quotient map. Let $(x_n)_{n\in\mathbb{N}}$ be a sequence of $X$ such that $\lim_{n\to\infty}x_n=x$ for some $x\in X$. We see that $\limsup \rho(x_n)\subseteq \rho(x)$. Let $z\in \limsup \rho(x_n)$. Since $(\rho(x_n))_{n\in\N}$ is a sequence in $\mathcal{L}$ and $\mathcal{L}$ is compact, there exists a subsequence $(\rho(x_{n_k}))_{k\in\N}$ of $(\rho(x_n))_{n\in\N}$, such that $\lim_{k\to\infty}\rho(x_{n_k})=D$ for some $D\in\mathcal{L}$ where $z\in D.$ Let $w\in D$ be such that $D=\pi(w)$. Since $x_{n_k}\in \rho(x_{n_k})$ for each $k\in\N$, and $\lim_{k\to\infty}x_{n_k}=x$, we have that $x\in \pi(w)$. Hence, $\rho(x)\cap\pi(z)\neq\emptyset$. Since $\rho(x)\in\mathcal{M}$, $\pi(z)\subseteq \rho(x)$. Thus, $z\in\rho(x)$ and $\limsup \rho(x_n)\subseteq \rho(x)$. Therefore, $\mathcal{M}$ is an upper semicontinuous decomposition of $X$, by Proposition~\ref{usc}. We have proved \textit{1}.

\smallskip

Let $x\in X$ be such that $\pi(x)\in\mathcal{N}$. We prove that $\pi(x)$ is terminal. Since $\pi(x)\in\NB(\mathcal{F}_1(X))$, $\pi(y)\subseteq \pi(x)$ for every $y\in\pi(x)$, by (\ref{defminimal}). Since $\pi(x)$ is minimal, $\pi(y)=\pi(x)$ for all $y\in\pi(x)$. Let $L$ be a subcontinuum of $X$ such that $\pi(x)\cap L\neq\emptyset$. Suppose that $\pi(x)\setminus L\neq\emptyset$. Let $y\in\pi(x)\setminus L$. Note that $\pi(y)\cap L\neq\emptyset$. We know that $\pi(y)=B(y)$, by Proposition~\ref{blockpi}. Thus, $L\subseteq \pi(y)$, by Remark~\ref{rem6676}. Since $\pi(y)=\pi(x)$, $L\subseteq \pi(x)$. Therefore, $\pi(x)$ is terminal. The proof of \textit{2} is completed.

\smallskip

We see \textit{3}. Suppose that $\mathcal{N}$ is not compact, that is, there exists a sequence $(\pi(x_n))_{n\in\mathbb{N}}$ in $\mathcal{N}$ such that $\lim_{n\to\infty}\pi(x_n)=D$ where $D\notin \mathcal{N}$. Since $\mathcal{L}$ is compact, $D\in\mathcal{L}$. Let $z\in D$ be such that $D=\pi(z)$. Since $D\notin \mathcal{N}$, there exists $w\in\pi(z)$ such that $\pi(w)\subseteq \pi(z)$ and $\pi(w)\neq\pi(z)$. Let $p\in \pi(z)\setminus \pi(w)$. Since $\Int(\pi(x_n))=\emptyset$ for each $n\in\N$, $\Int(\pi(z))=\emptyset$ and $\lim_{n\to\infty}\pi(x_n)=\pi(z)$, we have that $K=(\bigcup_{n\in\N}\pi(x_n))\cup \pi(z)$ is a compact such that $\Int(K)=\emptyset$ (see \cite[Corollary~25.4]{Willard}). Let $V$ be a nonempty open subsets of $X$ such that $V\cap K=\emptyset$. Since $p\notin\pi(w)$, there exists a continuum $L$ such that $p\in L$, $L\cap\pi(w)=\emptyset$ and $L\cap V\neq\emptyset$. Observe that $\langle X, V\rangle$ is an open subset of $2^X$ such that $L\in \langle X, V\rangle$. Since $\lim_{n\to\infty}\pi(x_n)=\pi(z)$ and $p\in\pi(z)$, there exists a sequence $(p_n)_{n\in\N}$ such that $p_n\in\pi(x_n)$ for each $n\in\N$, and $\lim_{n\to\infty}p_n=p$. Since $X$ has the property of Kelley, there exists a sequence $(L_n)_{n\in\N}$ in $\mathcal{C}(X)$ such that $p_n\in L_n$ for each $n$, and $\lim_{n\to\infty}L_n=L$. Hence, there is $m\in\N$ such that $L_n\in \langle X, V\rangle$ for each $n\geq m$. Note that we have both $L_n\cap \pi(x_n)\neq\emptyset$ and $L_n\cap (X\setminus \pi(x_n))\neq\emptyset$, for each $n\geq m$. Thus, $\pi(x_n)\subseteq L_n$, because $\pi(x_n)$ is terminal,  for each $n\geq m$ (see \textit{2} in this proposition). Therefore, $\pi(z)\subseteq L$. This contradicts the facts that $L\cap \pi(w)=\emptyset$ and $\pi(w)\subseteq \pi(z)$. We finish the proof of \textit{3}.

\smallskip

We prove \textit{4}. Since $\mathcal{N}$ is compact, $\bigcup\mathcal{N}$ is compact, by \cite[Exercise~11.5]{Illanes-Nadler}. We suppose that $\bigcup\mathcal{N}$ is not connected, that is, $\bigcup\mathcal{N}=A\cup B$ where $A$ and $B$ are nonempty compact disjoint subsets of $X$. Observe that if $x\in \bigcup\mathcal{N}$, then there exists $\pi(z)\in\mathcal{N}$ such that $x\in\pi(z)$. Hence, $\pi(x)\subseteq \pi(z)$ and, since $\pi(z)$ is minimal, $\pi(x)=\pi(z)$. We have that $\pi(x)\subseteq \bigcup\mathcal{N}$ for each $x\in \bigcup\mathcal{N}$. Furthermore, $\pi(x)$ is connected for every $x\in \bigcup\mathcal{N}$. Thus, $\pi(x)\subseteq A$ for each $x\in A$, and $\pi(y)\subseteq B$ for each $y\in B$. 

Note that $\mathcal{L}$ is a continuum and $\langle X\setminus B\rangle\cap\mathcal{L}$ is a proper open subset of $\mathcal{L}$. Hence, there exists a sequence $(\pi(x_n))_{n\in\mathbb{N}}$ in $\langle X\setminus B\rangle\cap\mathcal{L}$ such that $\lim_{n\to\infty}\pi(x_n)=\pi(x)$ for some $\pi(x)\in \mathcal{L}\setminus\langle X\setminus B\rangle$. Let $y\in \pi(x)\cap B$. Observe that $\pi(y)\subseteq B$. Let $(y_n)_{n\in\mathbb{N}}$ be a sequence such that $y_n\in\pi(x_n)$ for each $n\in\mathbb{N}$, and $\lim_{n\to\infty}y_n=y$. Since $\pi(y_n)\subseteq \pi(x_n)\subseteq X\setminus B$, $\pi(y_n)\cap A\neq\emptyset$ for each $n\in\mathbb{N}$. By Theorem~\ref{partition}, $\pi$ is a map and $\lim_{k\to\infty}\pi(y_n)=\pi(y)$. Thus, $\pi(y)\cap A\neq\emptyset$. This contradicts that $\pi(y)\subseteq B$ and $A\cap B=\emptyset$. Therefore, $\bigcup\mathcal{N}$ is a subcontinuum of $X$.
\end{proof}

\begin{theorem}\label{theopartition}
Let $X$ be a decomposable continuum with the property of Kelley such that $\NB(\mathcal{F}_1(X))$ is a continuum. Let $\mathcal{L}, \mathcal{M}$ and $\mathcal{N}$ be as (\ref{defL}) and (\ref{defMN}). Then, $\mathcal{N}=\mathcal{L}$ and $\mathcal{N}$ is a continuous decomposition of $X$.
\end{theorem}

\begin{proof}
In order to prove that $\mathcal{N}=\mathcal{L}$, we show first that $\bigcup\mathcal{N}=X$. We know that $\bigcup\mathcal{N}$ is a continuum, by Proposition~\ref{khig788tj} (\textit{4}). Suppose that $\bigcup \mathcal{N}\neq X$.
By Lemma~\ref{lemdyfh6y}, there exists $B\in\NB(\mathcal{F}_1(X))$ such that $B\cap \bigcup\mathcal{N}=\emptyset$. Note that $\pi(x)\subseteq B$ for each $x\in B$. Hence, there exists $z\in X$ such that $\pi(z)\in\mathcal{N}$ and $\pi(z)\subseteq B$. A contradiction. Therefore, $\bigcup\mathcal{N}=X$. 

We see that $\mathcal{N}=\mathcal{L}$. Let $x\in X$. Since $\bigcup\mathcal{N}=X$, there exists $z\in X$ such that $\pi(z)\in\mathcal{N}$ and $x\in\pi(z)$. Hence, $\pi(x)\subseteq \pi(z)$ and, since $\pi(z)$ is minimal, $\pi(x)=\pi(z)$. Thus, $\pi(x)\in\mathcal{N}$ and $\mathcal{N}=\mathcal{L}$.

Observe that if every $\pi(x)$ is minimal, then every $\pi(x)$ is maximal; i.e., $\mathcal{N}=\mathcal{M}$. Thus, $\mathcal{N}$ is a decomposition, by Proposition~\ref{khig788tj} (\textit{1}). Finally, we have that $\mathcal{N}$ is a continuous decomposition because $\pi$ is a map (see Proposition~\ref{usc} and Theorem~\ref{partition}).
\end{proof}

We know that $\mathcal{N}$ is a continuous decomposition of $X$, and $\mathcal{N}\subseteq \mathcal{C}(X)$. Let $T(\mathcal{N})$ denote the decomposition topology on $\mathcal{N}$ induced by the surjective map $q\colon X\to \mathcal{N}$ defined by $q(x)=\pi(x)$ for each $x\in X$, and let $T_V(\mathcal{N})$ denote the topology on $\mathcal{N}$ as a subspace of $\mathcal{C}(X)$. Then, by \cite[Theorem~13.10]{Nadler-Libro1}, $q$ is an open map, and $T(\mathcal{N})=T_V(\mathcal{N})$. Furthermore, since $q(x)\in \mathcal{C}(X)$ for each $x\in X$, $q$ is a monotone map.


\begin{theorem}\label{lafhsk}
Let $X$ be a decomposable continuum with the property of Kelley such that $\NB(\mathcal{F}_1(X))$ is a continuum. Let $\mathcal{N}=\{\pi(x) : x\in X\}$ be the continuous decomposition of $X$ and let $q\colon X\to \mathcal{N}$ be the quotient map. Then, $\NB(\mathcal{F}_1(X))$ is homeomorphic to $\NB(\mathcal{F}_1(\mathcal{N}))$, and $\mathcal{F}_1(\mathcal{N})\subseteq \NB(\mathcal{F}_1(\mathcal{N}))$.
\end{theorem}

\begin{proof}
Let $D\in\NB(\mathcal{F}_1(X))$. By (\ref{defminimal}), $\pi(x)\subseteq D$ for each $x\in D$. Hence, $D=\bigcup\{\pi(x) : x\in D\}$, for each $D\in \NB(\mathcal{F}_1(X))$. Thus, we have that $D=q^{-1}(q(D))$ for each $D\in\NB(\mathcal{F}_1(X))$. Since $q$ is an open monotone map, by \cite[Lemma~5.1]{Escobedo-Villanueva}, we have that
\begin{equation}\label{eq43re}
    q(D)\in \NB(\mathcal{F}_1(\mathcal{N}))\text{ for each }D\in\NB(\mathcal{F}_1(X)).
\end{equation}
Let $g\colon \NB(\mathcal{F}_1(X))\to\NB(\mathcal{F}_1(\mathcal{N}))$ be defined by $g(D)=q(D)$ for each $D\in \NB(\mathcal{F}_1(X))$. Note that $g$ is well defined, by (\ref{eq43re}). Furthermore, $g=2^{q}|_{\NB(\mathcal{F}_1(X))}$ where $2^q\colon 2^X\to2^\mathcal{N}$ is the induced map given by $2^q(A)=q(A)$ for each $A\in 2^X$. Since $2^q$ is a map, $g$ is a map (see \cite[p.188]{Illanes-Nadler}). 

We see that $g$ is one-to-one. Let $A, B\in \NB(\mathcal{F}_1(X))$ be such that $A\neq B$. Without loss of generality, we may suppose that there exists $x\in A\setminus B$. Hence, $\pi(x)\in q(A)$. If $\pi(x)\in q(B)$, then there exists $z\in B$ such that $\pi(z)=\pi(x)$. Since $\pi(z)\subseteq B$, $x\in B$. A contradiction. Thus, $\pi(x)\notin q(B)$. Therefore, $g(A)\neq g(B)$ and $g$ is one-to-one.

Let $B\in \NB(\mathcal{F}_1(\mathcal{N}))$. By \cite[Lemma~5.1]{Escobedo-Villanueva}, $q^{-1}(B)\in \NB(\mathcal{F}_1(X))$. Since $q$ is surjective, $q(q^{-1}(B))=B$. Thus, $g(q^{-1}(B))=B$ and $g$ is surjective. Since $\NB(\mathcal{F}_1(X))$ is compact and $\NB(\mathcal{F}_1(\mathcal{N}))$ is a Hausdorff space, $g$ is a homeomorphism (see \cite[Theorem~17.14]{Willard}). Therefore, $\NB(\mathcal{F}_1(X))$ is homeomorphic to $\NB(\mathcal{F}_1(\mathcal{N}))$.

Finally, if $D\in \mathcal{F}_1(\mathcal{N})$, then there exists $x\in X$ such that $D=\{\pi(x)\}$. Since $\pi(x)\in\NB(\mathcal{F}_1(X))$, $g(\pi(x))\in \NB(\mathcal{F}_1(\mathcal{N}))$. Therefore, $D\in \NB(\mathcal{F}_1(\mathcal{N}))$ and $\mathcal{F}_1(\mathcal{N})\subseteq \NB(\mathcal{F}_1(\mathcal{N}))$.
\end{proof}

\section{More about the property of Kelley}

In this section, the show some extra properties related to a continuum with the property of Kelley and its hyperspace of nonblockers. The main result is Theorem~\ref{7uhbcs4} when we show that if $X$ is decomposable with the property of Kelley, $\NB(\mathcal{F}_1(X))$ is a continuum and $\mathcal{F}_1(X)\subseteq\NB(\mathcal{F}_1(X))$, then $X$ is a aposyndetic.

\begin{lemma}\label{elinutil}
Let $X$ be a decomposable continuum with the property of Kelley. Let $A \in \mathcal{C}(X)$ and let $x,y \in X \setminus A$ be such that $\kappa_{X \setminus A}(x) \cap \kappa_{X \setminus A}(y)= \emptyset$. If  $\kappa_{X \setminus A}(x) \cap \Cl(\kappa_{X \setminus A}(y)) \neq \emptyset$, then $\kappa_{X \setminus A}(x) \subseteq \Cl(\kappa_{X \setminus A}(y))$.
\end{lemma}

\begin{proof}
Let $A \in \mathcal{C}(X)$ and let $x,y \in X \setminus A$ be such that $\kappa_{X \setminus A}(x) \cap \kappa_{X \setminus A}(y)= \emptyset$ and  $\kappa_{X \setminus A}(x) \cap \Cl(\kappa_{X \setminus A}(y)) \neq \emptyset$. Let $z\in \kappa_{X \setminus A}(x) \cap \Cl(\kappa_{X \setminus A}(y))$. Suppose that $\kappa_{X\setminus A}(x)\setminus\Cl(\kappa_{X\setminus A}(y))\neq\emptyset$. Let $w\in \kappa_{X\setminus A}(x)\setminus\Cl(\kappa_{X\setminus A}(y))$. Since $\{z,w\}\subseteq \kappa_{X \setminus A}(x)$, there exists a continuum 
$L$ such that $\{z,w\}\subseteq L$ and $L\cap A=\emptyset$. Let $U$ be an open subset of $X$ such that $w\in U$ and $U\cap \Cl(\kappa_{X\setminus A}(y))=\emptyset$. Note that $L\in\langle X\setminus A,U\rangle$ where $\langle X\setminus A,U\rangle$ is an open subset of $2^X$. Since $z\in\Cl(\kappa_{X \setminus A}(y))\setminus  \kappa_{X \setminus A}(y)$, there exists a sequence $(z_n)_{n\in\mathbb{N}}$ in $\kappa_{X\setminus A}(y)$ such that $\lim_{n\to\infty}z_n=z$. Since $X$ has the property of Kelley, there exists a sequence of subcontinua $(L_n)_{n\in \mathbb{N}}$ of $X$ such that $z_n\in L_n$ for each $n\in\N$, and $\lim_{n\to\infty}L_n=L$. Hence, there exists $k\in\N$ such that $L_i\in\langle X\setminus A,U\rangle$ for each $i\geq k$. Let $m>k$. Hence, $L_m\cap \kappa_{X \setminus A}(y)\neq\emptyset$ and $L_m\cap A=\emptyset$. Thus, $L_m\subseteq \kappa_{X \setminus A}(y)$. This contradicts that $L_m\cap U\neq\emptyset$ and $U\cap \Cl(\kappa_{X\setminus A}(y))=\emptyset$. Therefore, $\kappa_{X \setminus A}(x) \subseteq \Cl(\kappa_{X \setminus A}(y))$.
\end{proof}

\begin{lemma}\label{lempre1}
Let $X$ be a continuum and let $(A_n)_{n\in\mathbb{N}}$ be a sequence of pairwise disjoint subcontinua of $X$ such that $\lim_{n\to\infty}A_n=A$ for some $A\in\mathcal{C}(X)$. Then, there exists a countable set $D$ such that $D\subseteq \bigcup_{n\in\mathbb{N}}A_n$, $|D\cap A_i|=1$ for each $i\in\mathbb{N}$, and $A\subseteq \Cl(D)$. 
\end{lemma}

\begin{proof}
Let $E=\{e_n : n\in\mathbb{N}\}$ be a dense subset of $A$. For each $n\in\mathbb{N}$, let $(x_i^n)_{i\in\mathbb{N}}$ be a sequence such that $x_i^n\in A_i$ for each $i\in\mathbb{N}$, and $\lim_{i\to\infty}x_i^n=e_n$ for each $n\in\mathbb{N}$. Let $\mathcal{N}_1, \mathcal{N}_2 \ldots$ be a sequence of pairwise disjoint subsets of $\mathbb{N}$ such that $\mathcal{N}_n$ is infinite for each $n\in\mathbb{N}$ and $\mathbb{N}=\bigcup_{n\in\mathbb{N}}\mathcal{N}_n$. Let $D=\bigcup_{n\in\mathbb{N}}D_n$ where $D_n=\{x_i^n : i\in\mathcal{N}_n\}$. It is not difficult to see that $D$ satisfies all the conditions in our lemma. 
\end{proof}

\begin{lemma}\label{lemyfhdu87}
Let $X$ be a continuum and let $A$ be a proper subcontinuum of $X$ such that $X\setminus A$ is not connected by continua. Let $B\in 2^X\setminus\mathcal{C}(X)$ be such that $\Int(B)=\emptyset$ and $B\subseteq X\setminus A$. Suppose that $B$ satisfies the following conditions:
\begin{enumerate}
    \item $D\in\NB(\mathcal{F}_1(X))$ for every component $D$ of $B$;
    \item $\Cl(\kappa_{X\setminus A}(x))\cap \kappa_{X\setminus A}(x')=\emptyset$ for every $x,x'\in X\setminus A$ such that $\kappa_{X\setminus A}(x)\cap \kappa_{X\setminus A}(x')=\emptyset$ $\Cl(\kappa_{X\setminus A}(x))\cap B\neq\emptyset$ and $\Cl(\kappa_{X\setminus A}(x'))\cap B\neq\emptyset$; and
    \item $B\cap \kappa_{X\setminus A}(x)$ is a component of $B$ for each $x\in X\setminus A$ such that $\kappa_{X\setminus A}(x)\cap B\neq\emptyset$.
\end{enumerate}
Then, $B\in \mathcal{NB}(\mathcal{F}_1(X))$.
\end{lemma}

\begin{proof}
We prove that $X\setminus B$ is connected by continua. To this end, we show that for every $x\in X\setminus B$, there exists a subcontinuum $L_x$ of $X\setminus B$ such that $x\in L_x$ and $L_x\cap A\neq\emptyset$. Let $x\in X\setminus B$. If $x\in A$, then we define $L_x=A$ and clearly $x\in L_x$, $L_x\cap B=\emptyset$ and $L_x\cap A\neq\emptyset$. Hence, we suppose that $x\in X\setminus A$. Note that if $\Cl(\kappa_{X\setminus A}(x))\cap B=\emptyset$, then $L_x=\Cl(\kappa_{X\setminus A}(x))$ is a continuum satisfying the conditions. Thus, assume that $\Cl(\kappa_{X\setminus A}(x))\cap B\neq\emptyset$. Let $z\in \Cl(\kappa_{X\setminus A}(x))\cap B$. It is clear that $\Cl(\kappa_{X\setminus A}(z))\cap B\neq\emptyset$. Thus, by \textit{2}, $\kappa_{X\setminus A}(x)\cap \kappa_{X\setminus A}(z)\neq\emptyset$ and $\kappa_{X\setminus A}(x)=\kappa_{X\setminus A}(z)$. Observe that we showed that 
\begin{equation}\label{eq4b4}
    \Cl(\kappa_{X\setminus A}(x))\cap B=\kappa_{X\setminus A}(x)\cap B.
\end{equation}
Furthermore, $\kappa_{X\setminus A}(x)\cap B$ is a component of $B$, by \textit{3}. Since $B$ is not connected, there exists $b\in B$ such that  $b\in X\setminus \Cl(\kappa_{X\setminus A}(x))$ (see (\ref{eq4b4})). Hence, $X\setminus \Cl(\kappa_{X\setminus A}(x))\neq\emptyset$. Let $U=X\setminus \Cl(\kappa_{X\setminus A}(x))$ and let $D=\kappa_{X\setminus A}(x)\cap B$. We know that $D\in \NB(\mathcal{F}_1(X))$, by \textit{1}. Since $x\notin D$, there exists a subcontinuum $H$ of $X\setminus D$ such that $x\in H$ and $H\cap U\neq\emptyset$. Note that $H\cap A\neq\emptyset$. Let $L_x$ be the component of $H\cap \Cl(\kappa_{X\setminus A}(x))$ such that $x\in L_x$. If $L_x\cap A=\emptyset$, then there exists a subcontinuum $J$ of $H\setminus A$ such that $L_x\subseteq J$ and $J\neq L_x$, by \cite[Corollary~5.5]{Nadler-Libro1}. Thus, $J\subseteq \kappa_{X\setminus A}(x)$ contradicting that $L_x$ is a component of $H\cap \Cl(\kappa_{X\setminus A}(x))$; i.e., $L_x\cap A\neq\emptyset$. Since $H\cap D=\emptyset$ and $L_x\subseteq H$, $L_x\cap D=\emptyset$. Furthermore, $L_x\subseteq \Cl(\kappa_{X\setminus A}(x))$ and hence, $L_x\cap B=\emptyset$, by (\ref{eq4b4}). Therefore, for every $x\in X\setminus B$, there exists a subcontinuum $L_x$ of $X\setminus B$ such that $x\in L_x$ and $L_x\cap A\neq\emptyset$. 

If $x,y\in X\setminus B$, then there exist subcontinua $L_x$ and $L_y$ of $X\setminus B$ such that $x\in L_x, y\in L_y$, $L_x\cap A\neq\emptyset$ and $L_y\cap A\neq\emptyset$. Thus, $L_x\cup A\cup L_y$ is a proper subcontinuum of $X\setminus B$ containing both $x$ and $y$. Therefore, $X\setminus B$ is connected by continua and $B\in\NB(\mathcal{F}_1(X))$.
\end{proof}

\begin{theorem}\label{treokd8g}
Let $X$ be a decomposable continuum with the property of Kelley such that $\mathcal{F}_1(X)\subseteq \NB(\mathcal{F}_1(X))$ and $\NB(\mathcal{F}_1(X))$ is compact. If $A\in \mathcal{C}(X)$ is such that $\kappa_{X\setminus A}(w)$ is dense for some $w\in X\setminus A$, then there exists $L\in \NB(\mathcal{F}_1(X))$ such that $A\subseteq L$.
\end{theorem}

\begin{proof}
Let $A\in \mathcal{C}(X)$ and let $w\in X\setminus A$ be such that $\kappa_{X\setminus A}(w)$ is dense in $X$. If $A\in\NB(\mathcal{F}_1(X))$, then $L=A$. Hence, suppose that $A\in \mathcal{C}(X)\setminus \mathcal{NB}(\mathcal{F}_1(X))$. Let $y\in X\setminus A$ be such that $\Cl(\kappa_{X\setminus A}(y))\neq X$. We define 
\begin{equation}\label{eq5t}
    L=A\cup \{y\in X\setminus A : \Cl(\kappa_{X\setminus A}(y))\neq X\}.
\end{equation}
We see that $\Cl(\kappa_{X\setminus A}(y))\subseteq L$ for each $y\in L\setminus A$. Let $y\in L\setminus A$ and let $z\in \Cl(\kappa_{X\setminus A}(y))$. Note that if $z\in \Cl(\kappa_{X\setminus A}(y))\setminus A$, then $\kappa_{X\setminus A}(z)\subseteq \Cl(\kappa_{X\setminus A}(y))$, by Lemma~\ref{elinutil}. Hence, $\Cl(\kappa_{X\setminus A}(z))\neq X$ and $z\in L$. Therefore, $\Cl(\kappa_{X\setminus A}(y))\subseteq L$ for each $y\in L\setminus A$.

\smallskip

We prove that $L$ is a continuum. Since $\Cl(\kappa_{X\setminus A}(y))\cap A\neq\emptyset$ for each $y\in L\setminus A$, and $L=A\cup (\bigcup\{\Cl(\kappa_{X\setminus A}(y)) : y\in L\setminus A\})$, we have that $L$ is connected. We show that $L$ is closed. Let $(x_n)_{n\in\mathbb{N}}$ be a sequence of $L$ such that $\lim_{n\to\infty}x_n=x$ for some $x\in X$. We see that $x\in L$. Observe that if there exist $y_1, y_2, \ldots, y_k$ in $L\setminus A$ such that $$\mathcal{N}=\left\{n\in\mathbb{N} : x_n\in A\cup \Cl(\kappa_{X\setminus A}(y_1))\cup \Cl(\kappa_{X\setminus A}(y_2))\cup\cdots\cup\Cl(\kappa_{X\setminus A}(y_k))\right\}$$ is infinite, then $(x_n)_{n\in\mathcal{N}}$ converges to some point of $A\cup \Cl(\kappa_{X\setminus A}(y_1))\cup \Cl(\kappa_{X\setminus A}(y_2))\cup\cdots\cup\Cl(\kappa_{X\setminus A}(y_k))$. Hence, $x\in L$. Thus, by Lemma~\ref{elinutil}, we may assume that there exists a sequence $(y_n)_{n\in\mathbb{N}}$ of $L\setminus A$ such that $x_n\in\kappa_{X\setminus A}(y_n)$ for each $n\in\mathbb{N}$, and 
\begin{equation}\label{eq5t5}
    \Cl(\kappa_{X\setminus A}(y_i))\cap\kappa_{X\setminus A}(y_j)=\emptyset\text{ whenever that }i\neq j.
\end{equation}
Suppose that $x\notin L$. We will have a contradiction, showing that $X\in \NB(\mathcal{F}_1(X))$. Let $U_1,\ldots, U_m$ be open subsets of $X$ such that $X\in\langle U_1,\ldots, U_m \rangle$. Since $x\notin L$, $\kappa_{X\setminus A}(x)$ is dense. Thus, there exists a subcontinuum $M$ of $X\setminus A$ such that $x\in M$ and $M\in \langle U_1,\ldots, U_m \rangle$. Notice that $M\in \langle U_1,\ldots, U_m \rangle\cap \langle X\setminus A\rangle$. Since $X$ has the property of Kelley, there exists a sequence of subcontinua $(M_n)_{n\in\mathbb{N}}$ of $X$ such that $x_n\in M_n$ for each $n\in\mathbb{N}$, and $\lim_{n\to\infty}M_n=M$. Hence, there exists $k\in\N$ such that $M_n\in \langle U_1,\ldots, U_m \rangle\cap \langle X\setminus A\rangle$ for each $n\geq k$. Since $M_n\cap A=\emptyset$ and $x_n\in M_n$, we have that $M_n\subseteq \kappa_{X\setminus A}(y_n)$ for each $n\geq k$. Thus, $(M_n)_{n\geq k}$ is a sequence of pairwise disjoint subcontinua of $X$, by (\ref{eq5t5}). By Lemma~\ref{lempre1}, there exists a countable set $D\subseteq \bigcup_{n\geq k}M_n$ such that $|D\cap M_n|=1$ for each $n\geq k$ and $M\subseteq \Cl(D)$. We know that $\mathcal{F}_1(X)\subseteq \NB(\mathcal{F}_1(X))$. If $F$ is a finite subset of $D$, then $F\in\NB(\mathcal{F}_1(X))$, by Lemma~\ref{lemyfhdu87}. Let $R_m=(\bigcup_{k\leq i\leq m}M_i)\cap D$, for each $m\geq k$. Note that $R_n\in \NB(\mathcal{F}_1(X))$ for each $n\geq m$. Furthermore, since $\NB(\mathcal{F}_1(X))$ is compact, we may suppose that $\lim_{m\to\infty}R_m=Q$ for some $Q\in \NB(\mathcal{F}_1(X))$. Since $M\subseteq \Cl(D)$ and $\Cl(D)\subseteq Q$, $M\subseteq Q$. Thus, $Q\in \langle U_1,\ldots, U_m \rangle$. Therefore, $X$ is a limit point of $\NB(\mathcal{F}_1(X))$. This contradicts the compactness of $\NB(\mathcal{F}_1(X))$ and the fact that $X\notin \NB(\mathcal{F}_1(X))$. We have that $L$ is a subcontinuum of $X$.

Finally, $\kappa_{X\setminus A}(x)$ is dense and $L\cap \kappa_{X\setminus A}(x)=\emptyset$. Hence, $\Int(L)=\emptyset$. Furthermore, $\kappa_{X\setminus A}(z)$ is dense, for every $z\in X\setminus L$ (see definition of $L$ in (\ref{eq5t})). Therefore, $L\in \NB(\mathcal{F}_1(X))$ and $A\subseteq L$.
\end{proof}

\begin{theorem}\label{tr645463}
Let $X$ be a decomposable continuum with the property of Kelley such that $\mathcal{F}_1(X)\subseteq \NB(\mathcal{F}_1(X))$ and $\NB(\mathcal{F}_1(X))$ is compact. If $A\in \NB(\mathcal{F}_1(X))$, then $X\setminus A$ is connected by continua.
\end{theorem}

\begin{proof}
Let $x\in X\setminus A$. We prove that $X\setminus A=\kappa_{X\setminus A}(x)$. Suppose that there exists $y\in X\setminus A$ such that $\kappa_{X\setminus A}(y)\cap \kappa_{X\setminus A}(x)=\emptyset$. Note that $\kappa_{X\setminus A}(y)$ is dense, because $A\in\NB(\mathcal{F}_1(X))$. Since $\kappa_{X\setminus A}(x)$ is dense, there exists a sequence of subcontinua $(L_n)_{n\in\mathbb{N}}$ of $\kappa_{X\setminus A}(x)$ such that $L_n\subseteq L_{n+1}$ for each $n\in\mathbb{N}$, and $\lim_{n\to\infty}L_n=X$ \cite[Proposition~2.2]{Escobedo-Villanueva}. Since $\kappa_{X\setminus A}(y)\cap L_n=\emptyset$ and $\kappa_{X\setminus A}(y)$ is dense, we have that $\kappa_{X\setminus A}(y)\subseteq \kappa_{X\setminus L_n}(y)$; i.e., $\kappa_{X\setminus L_n}(y)$ is dense for each $n\in\mathbb{N}$. Hence, there exists  a continuum $K_n\in\NB(\mathcal{F}_1(X))$ such that $L_n\subseteq K_n$ for each $n\in\mathbb{N}$, by Theorem~\ref{treokd8g}. Since $\lim_{n\to\infty}L_n=X$, $\lim_{n\to\infty}K_n=X$. Thus, $X\in\NB(\mathcal{F}_1(X))$, a contradiction. Therefore, $X\setminus A=\kappa_{X\setminus A}(x)$ and $X\setminus A$ is connected by continua.
\end{proof}

\begin{theorem}\label{componentesabiertas}
Let $X$ be a decomposable continuum with the property of Kelley such that $\mathcal{F}_1(X)\subseteq \NB(\mathcal{F}_1(X))$ and $\NB(\mathcal{F}_1(X))$ is a continuum. If $W$ is a proper subcontinuum of $X$, then each component of $X\setminus W$ is open.
\end{theorem}

\begin{proof}
Suppose to the contrary that there exist a component $C$ of $X\setminus W$ and $p\in C$ such that $p\in \Cl(X\setminus C)$.

\begin{claim}\label{claim0}
$\Cl(\kappa_{X\setminus W}(p))\subseteq \Cl(X\setminus C)$.
\end{claim}
To prove the Claim~\ref{claim0}, we need to show that $\kappa_{X\setminus W}(p)\subseteq \Cl(X\setminus C)$.
Let $q\in \kappa_{X\setminus W}(p)$. Hence, there exists a subcontinuum $P$ of $X\setminus W$ such that $\{p,q\}\subseteq P$. 
Let $(p_n)_{n\in\mathbb{N}}$ be a sequence of $X\setminus C$ such that $\lim_{n\to\infty}p_n=p$. Since $X$ is compact, we may suppose that $(p_n)_{n\in\N}$ is a sequence in $X\setminus (W\cup C)$. Since $X$ has the property of Kelley, there exists a sequence of subcontinua $(P_n)_{n\in \mathbb{N}}$ of $X$ such that $p_n\in P_n$ for all $n$, and $\lim_{n\to\infty}P_n =P$. 
Notice that $\langle X\setminus W\rangle$ is an open subset of $2^X$, and $P\in \langle X\setminus W\rangle$. Hence, there exists $k\in\N$ such that $P_n\in \langle X\setminus W\rangle$ for each $n\geq k$. Since $p_n$ belongs to a different component to $C$ in $X\setminus W$, and $P_n\cap W=\emptyset$ for each $n\geq k$, we have that $P_n\cap C=\emptyset$ for all $n\geq k$. Hence, $P_n\subseteq X\setminus (W\cup C)$ for each $n\geq k$. Thus, $P\subseteq \Cl(X\setminus (W\cup C))\subseteq \Cl(X\setminus C)$. Therefore, $q\in \Cl(X\setminus C)$ and $\kappa_{X\setminus W}(p)\subseteq \Cl(X\setminus C)$. We have proved the Claim~\ref{claim0}.

\medskip

We prove that $2^{\Cl(\kappa_{X\setminus W}(p))}\subseteq \mathcal{NB}(\mathcal{F}_1(X))$. Let $F=\{w_1,\ldots,w_m\}$ be a finite subset of $\kappa_{X\setminus W}(p)$. By Claim~\ref{claim0}, for each $j\in\{1,\ldots,m\}$, there exists a sequence $(z^{j}_n)_{n\in\mathbb{N}}$ in $X\setminus C$ such that $\lim_{n\to\infty}z^{j}_n=w_j$. Since $W$ is compact, we may suppose that $(z^{j}_n)_{n\in\N}$ is a sequence in $X\setminus (C\cup W)$ for each $j\in\{1,\ldots,m\}$. Furthermore, observe that if $D$ is a component of $X\setminus W$ and $D\neq C$, then $\{n\in\N : z^{j}_n\in D\}$ is a finite set for each $j\in\{1,\ldots,m\}$. Hence, if necessary, we can take a subsequence $(z^{j}_{n_k})_{k\in\N}$ of $(z^{j}_n)_{n\in\N}$ for each $j\in\{1,\ldots,m\}$, such that $|\{z_{n_k}^{j} : k\in\mathbb{N}, j\in\{1,\ldots,m\}\}\cap D|\leq 1$ for each component $D$ of $X\setminus W$. Let $F_i=\{z^{1}_{n_i},\ldots,z^{m}_{n_i}\}$ for each $i\in\N$. Since $\mathcal{F}_1(X)\subseteq \NB(\mathcal{F}_1(X))$, $F_i\in \NB(\mathcal{F}_1(X))$ for each $i\in\N$, by Lemma~\ref{lemyfhdu87}. It is not difficult to see that $\lim_{i\to\infty}F_i=F$. Since $\NB(\mathcal{F}_1(X))$ is compact, $F\in \NB(\mathcal{F}_1(X))$. We have that $F\in \NB(\mathcal{F}_1(X))$ for every finite subset $F$ of $\kappa_{X\setminus W}(p)$. It is easy to see that the collection of all finite subsets of $\kappa_{X\setminus W}(p)$ is dense of $2^{\Cl(\kappa_{X\setminus W}(p))}$.
Hence, since $\NB(\mathcal{F}_1(X))$ is compact, $2^{\Cl(\kappa_{X\setminus W}(p))}\subseteq \NB(\mathcal{F}_1(X))$.

\medskip

Let $Q=W\cap \Cl(\kappa_{X\setminus W}(p))$. Since $Q\in 2^{\Cl(\kappa_{X\setminus W}(p))}$, $Q\in \NB(\mathcal{F}_1(X))$. Note that $\kappa_{X\setminus W}(p)\subseteq C$ and $C$ is a component of $X\setminus W$. Hence, $\Cl(\kappa_{X\setminus W}(p))\setminus \kappa_{X\setminus W}(p)\subseteq W$. Thus, $p\notin Q$ and $\kappa_{X\setminus Q}(p)\subseteq \kappa_{X\setminus W}(p)\subseteq C$. This clearly is a contradiction, because $\kappa_{X\setminus Q}(p)$ has to be dense and $C$ is not dense. Therefore, each component of $X\setminus W$ is open.
\end{proof}

\begin{theorem}\label{7uhbcs4}
Let $X$ be a decomposable continuum with the property of Kelley such that $\mathcal{F}_1(X)\subseteq \NB(\mathcal{F}_1(X))$. If $\NB (\mathcal{F}_1(X))$ is a continuum, then $X$ is aposyndetic.
\end{theorem}

\begin{proof}
Let $p\in X$. We show that $X$ is aposyndetic at $p$ with respect to each point of $X\setminus\{p\}$ (see Definition~\ref{defaposyndetic}). Let $q\in X\setminus \{p\}$. Since $\{q\}\in\NB(\mathcal{F}_1(X))$, $X\setminus\{q\}$ is connected by continua, by Theorem~\ref{tr645463}. Furthermore, there exists a sequence of subcontinua $(L_n)_{n\in\mathbb{N}}$ of $X\setminus \{q\}$ such that $X\setminus \{q\}=\bigcup_{n\in\mathbb{N}}L_n$, by Lemma~\ref{lemma0}. Since $X$ is a Baire space, there exists $k\in\mathbb{N}$ such that $\Int(L_k)\neq\emptyset$ \cite[Corollary~25.4]{Willard}. Since 
$X\setminus\{q\}$ is connected by continua, there exists a subcontinuum $R$ of $X\setminus\{q\}$ such that $p\in R$ and $R\cap L_k\neq\emptyset$. Since $X$ has the property of Kelley, there exists a continuum $W$ such that $L_k\cup R\subseteq\Int(W)\subseteq W\subseteq X\setminus\{q\}$, by Theorem~\ref{theokelley}. Thus, $p\in \Int(W)$ and $q\notin W$. Therefore,  $X$ is aposyndetic at $p$ with respect to $q$ and hence, $X$ is aposyndetic. 
\end{proof}

\section{The simple closed curve}

In this section, we present our main results. We show two characterizations in theorems \ref{f43eoa} and \ref{theo1q2q3q} of the simple closed curve.

\begin{theorem}\label{f43eoa}
Let $X$ be a decomposable continuum with the property of Kelley such that $\mathcal{F}_1(X)\subseteq \NB(\mathcal{F}_1(X))$. Then, $\NB(\mathcal{F}_1(X))$ is a continuum if and only if $X$ is a simple closed curve.
\end{theorem}

\begin{proof}
Suppose that $\NB(\mathcal{F}_1(X))$ is a continuum. Then, by Theorem~\ref{7uhbcs4}, $X$ is aposyndetic. Hence, $X$ is semi-locally connected, by \cite[Theorem 1.7.17]{Macias-Libro}. We prove that $X$ is locally connected. Let $p\in X$ and let $V$ be an open subset of $X$ such that $p\in V$. Since $X$ is semi-locally connected, there exists an open subset $U$ of $X$ such that $p\in U\subseteq V$ and $X\setminus U$ has a finite number of components. Let $L_1, L_2,\ldots ,L_k$ be the components of $X\setminus U$. Since $\{p\}\in\NB(\mathcal{F}_1(X))$, $X\setminus \{p\}$ is connected by continua, by Theorem \ref{tr645463}. Thus, for each $i\in\{2,\ldots,k\}$, there exists a subcontinuum $M_i$ of $X\setminus\{p\}$ such that $M_i\cap L_1\neq\emptyset$ and $M_i\cap L_i\neq\emptyset$. 

Let $$W=(L_1\cup L_2\cup \cdots\cup L_k)\cup (M_2\cup M_3\cup\cdots\cup M_k).$$ Observe that $W$ is a continuum and $p\notin W$. Let $J$ be the component of $X\setminus W$ such that $p\in J$. By Theorem~\ref{componentesabiertas}, $J$ is open. Furthermore, $J\subseteq X\setminus W\subseteq U\subseteq V$. Thus, $X$ is locally connected at $p$. Therefore, $X$ is a simple closed curve, by Theorem~\ref{locallyconnected}.

\medskip

Conversely, $\NB(\mathcal{F}_1(S^1))=\mathcal{F}_1(S^1)$. Therefore, $\NB(\mathcal{F}_1(S^1))$ is a continuum.
\end{proof}

\begin{theorem}\label{f43yheoa}
Let $X$ be a decomposable continuum with the property of Kelley. If\break $\NB(\mathcal{F}_1(X))$ is a continuum, then there exists a map $f\colon X\to S^1$ such that $f$ is open, monotone and $f^{-1}(z)$ is a terminal subcontinuum of $X$ for every $z\in S^1$.
\end{theorem}

\begin{proof}
We know that $\mathcal{N}=\{\pi(x) : x\in X\}$ is a continuous decomposition of $X$ such that $\pi(x)$ is a terminal subcontinuum of $X$ for every $x\in X$, by Proposition~\ref{khig788tj} and Theorem~\ref{theopartition}. Let $\rho\colon X\to\mathcal{N}$ be the quotient map. Since $\rho$ is monotone, $\mathcal{N}$ has the property of Kelley, by \cite[Theorem~4.3]{Wardle}. It is not difficult to see that $\mathcal{N}$ is decomposable. Furthermore, $\NB(\mathcal{F}_1(\mathcal{N}))$ is a continuum and $\mathcal{F}_1(\mathcal{N})\subseteq \NB(\mathcal{F}_1(\mathcal{N}))$, by Theorem~\ref{lafhsk}. Thus, $\mathcal{N}$ is a simple closed curve, by Theorem~\ref{f43eoa}. Let $g\colon \mathcal{N}\to S^1$ be a homeomorphism. It is clear that $f=g\circ \rho$ is a map from $X$ to $S^1$ monotone, open and $f^{-1}(z)$ is terminal for each $z\in S^1$. 
\end{proof}

The following result give us a partial answer to Question~\ref{Question}.

\begin{theorem}\label{theo1q2q3q}
Let $X$ be an hereditarily decomposable continuum with the property of Kelley. Then, $\NB(\mathcal{F}_1(X))$ is a continuum if and only if $X$ is a simple closed curve.
\end{theorem}

\begin{proof}
We assume that $\NB(\mathcal{F}_1(X))$ is a continuum. By Theorem~\ref{f43yheoa}, there exists a map $f\colon X\to S^1$ such that $f$ is open, monotone and $f^{-1}(z)$ is a terminal subcontinuum of $X$ for each $z\in S^1$. We show that $|f^{-1}(z)|=1$ for every $z\in S^1$. Let $z_0\in S^1$ and let $\alpha$ be an arc in $S^1$ such that $z_0\in \alpha$. Let $L=f^{-1}(\alpha)$. Since $f$ is monotone, $L$ is a continuum. Let $p,q\in L$ be such that $\alpha$ is irreducible between $f(p)$ and $f(q)$. We see that $L$ is irreducible between $p$ and $q$. Let $K$ be a subcontinuum of $L$ such that $p,q\in K$. Hence, $f(K)=\alpha$. Since $f^{-1}(z)$ is terminal for each $z\in \alpha$ and $f^{-1}(z)\cap K\neq\emptyset$, we have that $f^{-1}(z)\subseteq K$ for each $z\in\alpha$. Thus, $K=f^{-1}(\alpha)$ and $K=L$. Therefore, $L$ is irreducible.

Let $g=f|_{L}\colon L\to \alpha$. Since $f$ is open, $g$ is open, by \cite[Theorem~4.31]{Whyburn}. Furthermore, $g$ monotone and $g^{-1}(z)$ is terminal for each $z\in\alpha$. Thus, $D=\{x\in L : g^{-1}(g(x))=x\}$ is a dense subset of $L$, by \cite[Corollary~9, p.130]{Oversteegen-Tymchatyn}. Since $g$ is open, $\{g^{-1}(z) : z\in\alpha\}$ is a continuous decomposition of $L$ (see \cite[Theorem~4.32, p.130]{Whyburn}). Thus, $g^{-1}(g(x))=x$ for each $x\in L$. In particular, $|f^{-1}(z_0)|=1$. Therefore, $f$ is injective, and $f$ is a homeomorphism.

\medskip

It is clear that $\NB(\mathcal{F}_1(S^1))$ is a continuum. Therefore, we complete the proof of our theorem.
\end{proof}

The first author thanks La Vicerrector\'ia de Investigaci\'on y Extensi\'on de la Universidad Industrial de Santander y su Programa de Movilidad for financial support.

Escuela de Matem\'aticas, Facultad de Ciencias, Universidad Industrial de Santander, Ciudad Universitaria, Carrera 27 Calle 9, Bucaramanga, Santander, A.A. 678, COLOMBIA.

\textit{E-mail address:} \textbf{jcamargo@saber.uis.edu.co}

\smallskip

Escuela de Matem\'aticas, Facultad de Ciencias, Universidad Industrial de Santander, Ciudad Universitaria, Carrera 27 Calle 9, Bucaramanga, Santander, A.A. 678, COLOMBIA.

\textit{E-mail address:} \textbf{mayraferreira.ortiz@gmail.com}

\end{document}